\documentclass[12pt]{article}
\usepackage[english]{babel}
\usepackage{subfig}
\usepackage{hyperref,url, enumerate}
\usepackage[numbers,sort&compress]{natbib}
\usepackage{amsfonts}
\usepackage{mathrsfs} 
\usepackage{appendix}
\usepackage{mathtools}
\usepackage{amsmath}
\usepackage{amssymb}
\usepackage{amsthm}
\usepackage{bbm}
\usepackage{latexsym}
\usepackage{natbib}
\usepackage[mathscr]{euscript}
\usepackage{graphicx}
\usepackage{tikz}
\usepackage{pgfplots}
\usepackage{pgfplotstable}
\usepackage{fullpage}
\usepackage{comment}
\usetikzlibrary{arrows,positioning,chains,fit,shapes,calc}
\newtheorem{theorem}{Theorem}[section]
\newtheorem{lemma}[theorem]{Lemma}
\newtheorem{corollary}[theorem]{Corollary}
\newtheorem{proposition}[theorem]{Proposition}
\newtheorem{fact}[theorem]{Fact}

\newtheorem{claim}{Claim}[theorem]

\pgfplotsset{compat=1.8}

\def\NN{{\mathbb N}}  
\def\EE{{\mathbb E}}

\def\epsilon{\varepsilon}
\let\eps=\varepsilon

\renewcommand{\star}{\ast}

\title{On the Erd\H{o}s--S\'os conjecture for trees with bounded degree}
\date{}

\author{Guido Besomi, Mat\'ias Pavez-Sign\'e\footnote{MPS was supported by ANID	Doctoral scholarship ANID-PFCHA/Doctorado Nacional/2017-21171132.}, and Maya Stein\footnote{MS is also affiliated to Centro de Modelamiento Matem\'atico, Universidad de Chile, UMI 2807 CNRS. MS acknowledges support by  CONICYT + PIA/Apoyo a centros cient\'ificos y tecnol\'ogicos de excelencia con financiamiento Basal, C\'odigo AFB170001,
and by Fondecyt Regular Grant 1183080.}\\ \ \\
Departamento de Ingenier\'ia Matem\'atica\\ Universidad de Chile\\  
}

\begin{document}

\maketitle
\begin{abstract}
We prove the Erd\H os--S\'os conjecture for trees with bounded maximum degree and large dense host graphs. As a corollary, we obtain an upper bound on the multicolour Ramsey number of large trees whose maximum degree is bounded by a constant.
\end{abstract}
		
\section{Introduction} 
Given $k\in\mathbb N$, the famous Erd\H os--S\'os conjecture from 1964 (see~\cite{Erdos64}) states that every graph with average degree greater than $k-1$ contains all trees with $k$ edges. This conjecture is tight for every $k\in\mathbb N$, which can be seen by considering the complete graph on $k$ vertices. This graph has average degree  $k-1$ but it is too small to  contain any tree  with $k$ edges. A structurally different example is the {\it balanced} complete bipartite graph on $2k-2$ vertices (where by {\it balanced} we mean that the bipartition classes have equal sizes). This graph has average degree $k-1$ but does not contain the $k$-edge star. In order to obtain examples of larger order, one can consider the disjoint union of copies of the two {\it extremal graphs} we just described.

It is easy to see that the Erd\H os--S\'os conjecture is true for stars and  double stars (the latter are graphs obtained by joining the centres of two stars with an edge). A classical result of Erd\H os and Gallai~\cite{Erdos1959} implies that it also holds for paths. In the early 90's Ajtai, Koml\'os, Simonovits and Szemer\'edi announced a proof of the Erd\H os--S\'os conjecture for large~$k$. Nevertheless, many particular cases has been settled since then. For instance, Brandt and Dobson~\cite{bradob} proved that the Erd\H os--S\'os conjecture is true for graphs with girth at least $5$, and Sacl\'e and Wo\'zniak~\cite{sacwoz} proved it for $C_4$-free graphs. Goerlich and Zak~\cite{goerlich2016} proved the Erd\H os--S\'os conjecture for graphs of order $n=k+c$, where~$c$ is a given constant and $k$ is sufficiently large depending on $c$. More recently,   Rozho\v{n}~\cite{rohzon} gave an approximate version of the Erd\H os--S\'os conjecture for trees with linearly bounded maximum degree and dense host graph. Independently, the authors proved in~\cite{BPS1} a similar result but for trees with maximum degree bounded by $k^{\frac 1{67}}$ and dense host graphs. 

Given a positive integers $k$ and $\Delta$, let $\mathcal T(k,\Delta)$ denote the set of all trees $T$ with~$k$ edges and $\Delta(T)\le\Delta$. The main result of this paper is that the Erd\H os--S\'os conjecture holds for all trees whose maximum degree is bounded by a constant and whose size is linear in the order of the host graph. 

\begin{theorem}\label{thm:main}
For all $\delta>0$ and $\Delta\in \mathbb N$, there is $n_0\in\NN$ such that for each $k, n\in\mathbb N$ with $n\ge n_0$ and $n\ge k\ge \delta n$, and for each $n$-vertex graph $G$ the following holds. If $G$ satisfies $d(G)> k-1$, then $G$ contains every tree $T\in\mathcal T(k,\Delta)$.\end{theorem}

Our proof of Theorem~\ref{thm:main} splits into two cases. If $G$ is connected and considerably larger than $k$, we proceed as follows. After regularising $G$ we 
 inspect the components of the reduced graph, at least one of which has to have large average degree. If this component is large enough, then we can show it is either bipartite or contains a useful matching structure, and can embed any given tree $T\in\mathcal T(k,\Delta)$ using regularity and tools from~\cite{BPS1}. Otherwise,  the reduced graph is a union of graphs corresponding to the description given in the first paragraph of the Introduction, that is, graphs which are almost complete and of size roughly $k$ or balanced almost complete bipartite graphs of size roughly $2k$. In that case we use an edge of $G$ to connect two components and embed $T$ there.

If, on the other hand, the order of the host graph is very close to $k$,  if the host graph is close to being a  bipartite graph of order $2k$, or if the host graph is the disjoint union of such graphs, then a different approach is needed. To take care of these cases, we prove the following result, Theorem~\ref{smalltheorem}.

This theorem might be of independent interest as it greatly improves the main result from~\cite{goerlich2016} for bounded degree trees. Note that given a graph $G$ with $d(G)>k-1$,
a standard argument\footnote{We iteratively remove from $G$ vertices of degree less than $\frac k2$. This will not affect the average degree, and result in the desired minimum degree, unless we end up removing all vertices. However, that cannot happen, as then $|E(G)|<\frac k2\cdot n\le d(G)\cdot  \frac{n}2$, a contradiction.} shows that $G$ has a subgraph of minimum degree  $\delta(G)\ge\frac k2$ that preserves the average degree. So, since in the Erd\H os-S\'os conjecture and all our theorems, we are looking for subgraphs, we may always assume that in addition to the average degree condition, $G$ fulfills a minimum degree condition. (In particular, this is assumed in Theorem~\ref{smalltheorem}.)

Given $\beta>0$, we say that a graph $H$ is $\beta$-bipartite  if there is a partition $V(H)=A\cup B$ such that $e(A)+e(B)\le\beta e(H).$

\begin{theorem}\label{smalltheorem}
For each $k, \Delta \in\mathbb N$ and each graph $G$  with  $d(G)>k-1$ and  $\delta(G)\ge \frac k2$ the following holds.
\begin{enumerate}[(a)]
\item If $k\ge 10^{6}$ and  $|G|\le (1+10^{-11})k$  then $G$ contains each tree $T\in\mathcal T(k,\frac{\sqrt k}{1000})$.
\item If $k\ge 8\Delta^2$ and $G=(A,B)$ is $\frac 1{50\Delta^2}$-bipartite with  $|A|,|B|\le (1+\frac 1{25\Delta^2})k$  then $G$ contains each tree $T\in\mathcal T(k,\Delta)$. 
\end{enumerate}
 \end{theorem}
As a third result we prove an approximate version of the Erd\H os--S\'os conjecture for trees with linearly bounded maximum degree and dense host graph; this was independently proved by Rozho\v{n}~\cite{rohzon}. 

\begin{theorem}\label{thm:es_app}
For all $\delta\in (0,1)$ there are $n_0\in\NN$ and $\gamma\in (0,1)$ such that for each $k$ and for each $n$-vertex graph $G$  with $n\ge n_0$ and $n\ge k\ge \delta n$  the following holds. If $G$ satisfies $d(G)\ge (1+\delta)k$, then $G$ contains every tree $T\in\mathcal T(k,\gamma k)$. 
\end{theorem}

Finally, let us briefly mention a well-known consequence of the Erd\H os--S\'os conjecture in Ramsey theory. Given an integer $\ell\ge 2$ and a graph $H$, the $\ell$-colour Ramsey number $r_\ell(H)$ of $H$ is the smallest $n\in\mathbb N$ such that every $\ell$-colouring of the edges of $K_n$ yields a monochromatic copy of $H$. In 1973, Erd\H os and Graham conjectured~\cite{ErdosGraham} that every tree $T$ with $k$ edges satisfies
	\begin{equation}\label{ramsey:trees} r_\ell(T)=\ell(k+1)+O(1),\end{equation}
	and they established the lower bound $r_\ell(T)>\ell(k+1)+1$ for large enough $\ell$ satisfying $\ell\equiv1\mod k$. Erd\H os and Graham also observed that the upper bound in~\eqref{ramsey:trees} would follow from the Erd\H os--S\'os conjecture. Indeed, for $n\ge \ell(k-1)+2$  note that the most popular colour in any $\ell$-colouring of $K_n$ has at least $\frac{1}{\ell}\binom{n}2$ edges and thus average degree at least $\frac{n-1}{\ell}>k-1$. So the Erd\H os--S\'os conjecture would imply that the most popular colour contains a copy of every tree with $k$ edges. Therefore, from Theorem~\ref{thm:main} we deduce the following result.
	\begin{corollary}\label{cor:Ramsey}
	For all $\ell\ge 2$, $\Delta\in\mathbb N$ there exists $k_0\in\mathbb N$ such that for every $k\ge k_0$ and every tree $T\in\mathcal T(k,\Delta)$ we have
		$r_\ell(T)\le \ell(k-1)+2.$
\end{corollary}
We remark that in Corollary~\ref{cor:Ramsey} one can actually find a copy of every tree $T\in\mathcal T(k,\Delta)$ in the same colour, at the same time.

The paper is organised as follows.  After some preliminaries in Section~\ref{sec:prelim}, we prove Theorem~\ref{thm:main}  in Section~\ref{sec:smaaall}. That Section also contains the proof of Theorem~\ref{smalltheorem}, more precisely, Theorem~\ref{smalltheorem} follows directly from Propositions~\ref{prop:small:bip} and~\ref{prop:small} stated and proved in that section.
We finally prove  Theorem~\ref{thm:es_app} in Section~\ref{sec:linear}.

\section{Preliminaries}\label{sec:prelim}
\subsection{Notation}
For $\ell\in\mathbb{N}$, we write $[\ell]$ for the discrete interval $\{1,\ldots,\ell\}$. We write $a\ll b$ to indicate that  given a constant $b$, constant $a$ is chosen significantly smaller. The explicit value for such~$a$ can be calculated from the proofs. Also, we write $a=b\pm c$ if $a\in[b-c,b+c]$.

Given a graph $H$, write $|H|=|V(H)|$  and $e(H)=|E(H)|$. Let $\delta(H)$, $d(H)$ and $\Delta(H)$ denote the minimum, average and maximum degree of $H$, respectively.  As usual, $\deg_H(x)$ denotes the degree of a vertex $x\in V(H)$, and we write $N_H(x)$ for its neighbourhood in $H$,   $N_H(x,S)=N_H(x)\cap S$ for its neighbourhood in $S\subseteq V(H)$ and $\deg_H(x,S)$ for the respective degree. For two sets $X,Y\subseteq V(H)$, we write $E_H(X,Y)$ for the family of edges $xy\in E(H)$ with $x\in X$ and $y\in Y$ and set $e_H(X,Y):=|E_H(X,Y)|$. Note that edges lying in the intersection of $X$ and $Y$ are counted twice.
In all of the above, we omit the subscript~$H$ if it is clear from the context. 
Given $U\subset V(H)$ we write $H[U]$ for the graph induced in~$H$ by the vertices in $U$, and we say a vertex $x$ {\em sees} $U$ if it has at least one neighbour in $U$. 

Given a collection of sets $\mathcal{F}$, we write $\bigcup \mathcal{F}$ for the union of all members of $\mathcal{F}$. If $\mathcal G$ is a collection of graphs, then $\bigcup\mathcal G$ denotes the graph which is the union of all graphs in $\mathcal G$.

\subsection{Regularity Lemma}\label{sec:regul}
Let us fix two parameters $\eps,\eta\in(0,1)$. Let $H=(A,B;E)$ be a bipartite graph with density $d(A,B):=\frac{e(A,B)}{|A||B|}$. We say that the pair $(A,B)$ is {\em $\varepsilon$-regular} if 
$$|d(X,Y)-d(A,B)|<\varepsilon$$
for all $X\subseteq A$ and $Y\subseteq B$, with $|X|>\varepsilon|A|$ and $|Y|> \varepsilon |B|$.
Furthermore, we say that $(A,B)$ is $(\eps,\eta)$-regular if $(A,B)$ is $\eps$-regular and $d(A,B)\ge\eta$. Given an $\varepsilon$-regular pair $(A,B)$, with density $d$, we say that a subset $X\subseteq A$ is \textit{$\varepsilon$-significant} if $|X|> \varepsilon |A|$ (analogously for subsets of $B$). A vertex $x\in A$ is called \textit{$\varepsilon$-typical} to a significant set $Y\subseteq B$ if $\deg(x,Y)> (d-\varepsilon)|Y|$, and similar for a vertex $x\in B$. We will write just \textit{regular}, \textit{significant} or \textit{typical} if $\varepsilon$ is clear from the context.

Regular pairs behave like a typical random graph of the same edge density. For instance, almost every vertex is typical to any given significant set, and regularity is inherited by subpairs. Let us state these well-known facts in a precise form (see~\cite{regu} for a proof). 
\begin{fact}\label{fact:1}Let $(A,B)$ be an $\varepsilon$-regular pair with density $d$. Then the following holds:
	\begin{enumerate}[(i)]
		\item For any $\varepsilon$-significant $Y\subseteq B$, all but at most $\varepsilon|A|$ vertices from $A$ are $\varepsilon$-typical to $Y$.
		\item\label{fact:1,2} Let $\alpha\in (0,1)$. For any subsets $X\subseteq A$ and $Y\subseteq B$, with $|X|\ge\alpha|A|$ and $|Y|\ge\alpha|B|$, the pair $(X,Y)$ is $\frac{2\varepsilon}{\alpha}$-regular with density $d\pm\varepsilon$.
	\end{enumerate}
\end{fact}
\noindent
Given a graph $G$, we say that a vertex partition $V(G)=V_1\cup\ldots\cup V_\ell$ is $(\varepsilon,\eta)$-regular if 
\begin{enumerate}
	\item $|V_1|=|V_2|=\ldots=|V_\ell|$;
	\item  $V_i$ is independent for all $i\in[\ell]$; and
	\item for all $1\le i<j\le \ell$, the pair $(V_i,V_j)$ is $\varepsilon$-regular with density either $d(V_i,V_j)\ge\eta$ or $d(V_i,V_j)=0$.
\end{enumerate}
Szemer\'edi's regularity lemma~\cite{Sze78} states that every large graph has an almost spanning subgraph that admits a regular partition. We will use the following version (see for instance~\cite{regu}).
\begin{lemma}[Regularity lemma]\label{reg:deg}
	For all $\varepsilon>0$ and $m_0\in\NN$ there are $N_0, M_0$ such that the following holds for all $\eta\in[0,1]$ and $n\ge N_0$. Any $n$-vertex graph $G$ has a subgraph $G'$, with $|G|- |G'|\le \varepsilon n$ and $\deg_{G'}(x)\ge \deg_G(x)-(\eta+\varepsilon)n$ for all $x\in V(G')$, such that $G'$ admits an $(\varepsilon,\eta)$-regular partition $V(G')=V_1\cup\ldots\cup V_\ell$, with $m_0\le \ell\le M_0$.
\end{lemma}

The {\em $(\varepsilon,\eta)$-reduced graph} $\mathscr R$ corresponding to the $(\varepsilon,\eta)$-regular partition that is given by Lemma~\ref{reg:deg} has vertex set $V(\mathscr R)=\{V_i:i\in[\ell]\}$, called \textit{clusters}, and  an edge $V_iV_j$ for each $i,j$ with $d(V_i,V_j)\ge\eta$. We use calligraphic letters to refer to the reduced graph, or to subsets of its vertex set. Moreover, given $\mathscr C\subseteq V(\mathscr R)$, we write $|\mathscr C|$ for the number of clusters in $\mathscr C$. In contrast, we write $|\bigcup \mathscr C|$ for the number of vertices of the subgraph $\bigcup \mathscr C$ of $G$. Now we state some useful facts about the reduced graph (see~\cite{regu} for a proof). 

\begin{fact}\label{fact:3}Let $G$ be a $n$-vertex graph and let $\mathscr R$ be an $(\eps,\eta)$-reduced graph of $G$. Then the following holds.
	\begin{enumerate}[(i)]
		\item\label{fact:3,1}Given a cluster $C\in V(\mathscr R)$ we have
		\[\deg_{\mathscr R}(C)\ge \frac{1}{|C|}\sum_{v\in C}\deg(v)\cdot\frac{|\mathscr R|}n. \]
		In particular, summing over all clusters we have $d(\mathscr R)\ge d(G)\cdot \frac{|\mathscr R|}{n}$.
	\item\label{fact:3,2} Let $\mathscr{Y}$ be a collection of significant sets of clusters in $\mathscr R$ and let $C\in V(\mathscr R)$. Then 
	\[|\{Y\in \mathscr{Y}: v\text{ is typical to }Y\}|\ge (1-\sqrt{\eps})|\mathscr{Y}|\]
	for all but at most $\sqrt{\eps}|C|$ vertices $v\in C$.
	\end{enumerate}
\end{fact}
We close this subsection with a well-known lemma that illustrates why regularity is so useful for embedding trees. It states that a tree will always fit into a regular pair, if the  tree is small enough (but it may still be linear in the size of the pair). A proof can be found for instance in~\cite{AKS,BPS1}.
\begin{lemma}\label{lem:T1}
	Let $0<\beta\le\varepsilon\le \tfrac{1}{25}$. Let $(A,B)$ be a $(\varepsilon,5\sqrt{\varepsilon})$-regular pair with $|A|=|B|=m$, and let $X\subseteq A, Y\subseteq B, Z\subseteq A\cup B$ be such that $\min\{|X\setminus Z|, |Y\setminus Z|\}>\sqrt \varepsilon m$. \\
	Then any tree $T$ on at most $\beta m$ vertices can be embedded into $(X\cup Y)\setminus Z$.
	Moreover, for each $v\in V(T)$ there are at least $2\varepsilon m$ vertices from $(X\cup Y)\setminus Z$ that can be chosen as the image of $v$.
\end{lemma}

\subsection{Trees}\label{sec:trees}
Let us give some notation for trees. We will write $(T,r)$ for a tree $T$ rooted at $r\in V(T)$. Given any rooted tree $(T,r)$ and $x,y\in V(T)$, we say that $x$ is {\it below} $y$ (resp. $y$ is {\it above} $x$) if $y$ lies on the unique path from $x$ to $r$ (our trees grow from the top to the bottom). If in addition, $xy\in E(T)$,  we say $x$ is a {\it child} of $y$, and $y$ is the {\it parent} of $x$. 

The following lemma  allow us to find a cut vertex which splits the tree into connected components of convenient sizes. See~\cite{BPS1, 2k3:2016, RS19a} for other variants and a proof.

\begin{lemma}\label{cut_off}
	For all $0<\gamma\le 1$ and for all $k\ge \frac {200}\gamma$, any given tree $T$ with $k$ edges has a  subtree $(T^*,t^*)$ such that
	\begin{enumerate}[(i)]
		\item $\frac{\gamma k}2\le |V(T^*)|\le\gamma k$; and \label{Tstarsmall}
		\item every component of $T-T^*$  is adjacent to $t^*$.\label{compo}
	\end{enumerate}
	\end{lemma}

			

	
A \textit{bare path} in a tree is a path all whose internal vertices have degree  $2$ in the tree.  The next lemma has been extensively used in the literature of tree embeddings. It states that the structure of any given tree satisfies a certain dichotomy. Namely, each tree contains either a large number of leaves or a large number of bare paths of some fixed constant length (we refer to~\cite{Krivilevich, Montgomery} for a more general statement and a proof, and note that here,  the length of a path is its number of edges).
\begin{lemma}\label{bare}Let $\ell>2$ and let $T$ be a tree. Then either $T$ has at least $|T|/4\ell$ leaves or it has at least $|T|/4\ell$ vertex disjoint bare paths, each of length $\ell$. 
\end{lemma}

Another well-known fact we shall use in our proof is the following. One can prove it by rooting the tree at any vertex in the smaller bipartition class, and comparing the number of vertices in a odd level to the number of vertices in the preceding level.
\begin{fact}\label{thisisafact}
Let $T$ be a tree with bipartition $V(T)=C\cup D$ and maximum degree $\Delta(T)\le \Delta$. Then 
$\min\{|C|,|D|\}\ge \frac k\Delta.$
\end{fact}

\subsection{Tree embeddings}\label{sec:embeddings}
A greedy argument shows that every $k$-edge tree can be embedded into any graph of minimum degree at least $k$. We give two lemmas that generalise this simple observation.

\begin{lemma}\label{lem:greedy}
	Let $\Delta,h,k\in\NN$, let $(T,r)$ be a tree with $k-h$ edges and $\Delta(T)\le \Delta$, and let $G$ be a graph satisfying
	\begin{enumerate}[(i)]
		\item $\delta(G)\ge\Delta+h$;
		\item there are at most  $h$ vertices $x\in V(G)$ with $\deg(x)< k$. 
	\end{enumerate}
	Then $T$  can be embedded in $G$. Moreover, any vertex $v$ of $G$ can be chosen as the image of~$r$.
	
\end{lemma}
\begin{proof}
	We construct an embedding $\phi$ as follows. We set $\phi(r):=v$. Since $\deg(v)\ge \Delta+h$, we can embed each neighbour of $r$ into a neighbour of $v$ that has degree at least $k$. Since $T$ has $k-h$ vertices, we can then embed the rest of $T$ levelwise using only vertices of degree at least $k$ at each step. 
\end{proof}

Observe that  for $h=0$ Lemma~\ref{lem:greedy} recovers the greedy procedure  we mentioned above.

\noindent If the host graph $G$ is bipartite, 
 one can relax the minimum degree condition for one side of the bipartition of $G$. We leave the proof of the following lemma to the reader.
\begin{lemma}\label{lem:greedy3}
	Let $\Delta,h,k_1,k_2\in\NN$, let $(T,r)$ be a tree with colour classes  $C,D$ of sizes $k_1-h$ and $k_2-h$, respectively, and  $\Delta(T)\le\Delta$. Let $G=(A,B)$ be a bipartite graph  such that
	\begin{enumerate}[(i)]
		\item $\delta(G)\ge\Delta+h$;
		\item there are at most $h$ vertices $a\in A$ with $\deg(x)< k_2$;
		\item there are at most $h$ vertices $b\in B$ with $\deg(x)< k_1$.  
	\end{enumerate}
	Then $T$  can be embedded into $G$ with $C$ going to $A$ and $D$ going to $B$. Moreover, if $r\in C$ (resp. $D$), then any vertex $a\in A$ (resp. $b\in B$) can be chosen as the image of $r$.
	\end{lemma}

	\subsection{Matching lemma}

Later on we will need the following lemma on matchings in graphs with large minimum degree. This lemma is a slight variation of Lemma~5.7  from~\cite{BPS1}.
\begin{lemma}\label{lem:matching2}
	Let $\varepsilon,\eta\in(0,1)$,  let $t,\ell\in\mathbb{N}$, and let $G$ be a graph on $n\ge 2t+\ell$ vertices with $\delta(G)\ge t+\ell$ which has an $(\varepsilon,\eta)$-regular partition into $\ell$ parts. Then $G$ has a subgraph $G'$ with $|G'|\ge n-\ell$ that admits a $(5\varepsilon,\eta-\varepsilon)$-regular partition with $2\ell$ parts whose corresponding reduced graph $\mathscr R$ contains a matching $\mathscr M$ and an independent family of clusters~$\mathscr I$, disjoint from~$\mathscr M$, such that	
	\begin{enumerate}[(i)]
		\item $\bigcup V(\mathscr M)\cup V(\bigcup\mathscr{I})=V(G')$;\label{cover}
		\item $|\bigcup V(\mathscr M)|\geq 2t$; and\label{size}
		\item there is a  partition $V(\mathscr M)=\mathscr V_1\cup \mathscr V_2$ such that $N_{\mathscr R}(\mathscr I)\subseteq \mathscr V_1$  and  every edge in $\mathscr M$ has one endpoint in $\mathscr V_1$ and one endpoint in $\mathscr V_2$.\label{partition}
	\end{enumerate}
\end{lemma}

\section{Trees with constant maximum degree }\label{sec:smaaall}
In this section we work towards the proof of our main result, Theorem~\ref{thm:main}, and along the way, we prove Theorem~\ref{smalltheorem}. This latter theorem follows directly from Propositions~\ref{prop:small:bip} and~\ref{prop:small}. These are proved in  Sections~\ref{sec:almostcomplete} and~\ref{sec:almostcompletebip}, respectively. In Section~\ref{sec:robust} we use a regularity approach and results from~\cite{BPS1} to cover the case when the host graph is significantly larger than the tree.
Finally, in Subsection~\ref{sec:proof_main}, we put everything together to prove Theorem~\ref{thm:main}.

\subsection{Almost complete bipartite graphs}\label{sec:almostcomplete}

Recall that $H$ is $\beta$-bipartite if at least a $(1-\beta)$-fraction of its edges lie between $A$ and $B$. 

\begin{proposition}\label{prop:small:bip}
Let $k, \Delta\in \mathbb N$ such that $k\ge 8\Delta^2$. Let $G=(A,B)$ be a $\frac 1{50\Delta^2}$-bipartite graph, with  $|A|,|B|\le (1+\frac 1{25\Delta^2})k$,  $d(G)>k-1$ and $\delta(G)\ge \frac k{2}$. Then $G$ contains each tree $T\in\mathcal T(k,\Delta)$. 
\end{proposition}

\begin{proof} 
Set $\eps:= \frac 1{25\Delta^2}$ and  write $n=|V(G)|$. Then, $n\le (1+\eps)2k$. Since $G$ is $\frac \eps 2$-bipartite, 
	we know that $e(A,B)\ge (1-\eps)\frac{kn}{2}.$
	Suppose that $|B|\ge \frac n2\ge |A|$. Then
	\begin{equation}\label{avg:A}\frac{1}{|A|}\sum_{a\in A}\deg(a,B)>\frac{(1-\eps)kn}{2|A|}\ge (1-\eps)k,\end{equation}
	and thus $|B|\ge(1-\eps)k$. Furthermore, since $n=|A|+|B|$, we have
	$$|A||B|\ge e(A,B)\ge (1-\eps)\frac{kn}{2}\ge (1-\eps)k\sqrt{|A||B|},$$
	and thus, the fact that $|B|\le (1+\eps)k$ implies that
	$|A|\ge\frac{(1-\eps)^2}{1+\eps}k\ge (1-3\eps)k.$
	Now we can give a lower bound for the average degree from $B$ to $A$ by using the first inequality from~\eqref{avg:A} and the fact that $n=|A|+|B|$ to calculate
	\begin{equation}\label{avg:B}\frac{1}{|B|}\sum_{b\in B}\deg(b,A)> (1-\eps)\frac{k}{2}\Big(1+\frac{|A|}{|B|}\Big)\ge
	\frac{1-\eps}{2}\Big(1+\frac{1-3\eps}{1+\eps}\Big)k
	\ge(1-4\eps)k.
	\end{equation}
	Using Lemma~\ref{lem:con2} with $f_A(a)=\deg(a,B)$ for $a\in A$, $t_A=(1-\eps)k$ and $\eps_A=4\eps$, and with $f_B(b)=\deg(b,A)$ for $b\in B$, $t_B=(1-3\eps)k$ and $\eps_B=9\eps$, we see that 
 all but at most $2\sqrt\eps|A|$ vertices from $A$ have degree at least $(1-2\sqrt\eps)k$ to $B$, and all but at most $3\sqrt\eps|B|$ vertices from $B$ have degree at least $(1-3\sqrt\eps)k$ to $A$.
	Let $A_0$ and $B_0$ be the set of vertices of low degree in $A$ and $B$ respectively, and let $H$ be the bipartite graph induced by $A'=A\setminus A_0$ and $B'=B\setminus B_0$. Then the minimum degree of $H$ is at least $(1-5\sqrt{\eps})k$. Now, given a tree $T\in\mathcal T(k,\Delta)$, if  $V(T)=C\cup D$ is its natural bipartition, Fact~\ref{thisisafact} implies that
	\[\max\{|C|,|D|\}\le \Big(1-\frac 1\Delta\Big)k\le (1-5\sqrt{\eps})k,\]
	and therefore, by Lemma~\ref{lem:greedy3}, we can embed $T$ in $H$.%
	\end{proof}

\subsection{Almost complete graphs}\label{sec:almostcompletebip}

Now we turn to the non-bipartite case. In this case we can embed trees with maximum degree in $o(\sqrt k)$. As a first step we will embed a small but linear size subtree $T^*\subseteq T$ trying to fill up as many low degree vertices of $G$ as possible.  We can then use the following result to embed the leftover vertices from $T-T^*$.

\begin{lemma}[\cite{2k3:2016}, Lemma 4.4]\label{lem:mindeg1}Let $0<\nu<\tfrac{1}{200}$, let $k\in\mathbb{N}$ and let $H$ be a $k+1$-vertex graph with $\delta (H)\ge (1-2\nu)k$, and let $v\in V(H)$ be a vertex of degree $k$. If $(T,r)$ is a tree with at most $k$ edges such that every vertex is adjacent to at most $\nu k/2$ leaves, then $T$ can be embedded in $H$ and any vertex in $H-v$ can be chosen as the image of~$r$. \end{lemma}

\begin{proposition}\label{prop:small}
Let $k\ge 10^{6}$ and let $G$ be a graph on $n\le (1+10^{-11})k$ vertices such that $d(G)>k-1$ and  $\delta(G)\ge \frac k2$. Then $G$ contains every tree $T\in\mathcal T(k,\frac{\sqrt k}{1000})$.
\end{proposition}

\begin{proof}
Given $G$ and $k$, set $\varepsilon:=10^{-11}$ and note that necessarily, $n>k$. Moreover, for the complement $\bar G$ of~$G$, we have that $d(\bar G)<n-k$. Thus,
	\begin{equation}\label{G:size}2e(\bar{G})<n(n-k)\le (1+\varepsilon )k\cdot \varepsilon k\le 2\varepsilon k^2.
	\end{equation}
	
	Let $X$  be  the set of all  vertices of $G$ having degree at most $\lfloor(1-\sqrt\eps)k\rfloor$ in~$G$, and let $Y$ be the set of all vertices of $G$ having degree at least $k$ in $G$. 
Since $\deg(v)\le k-1$ for all $v\not\in Y$, we have that
$$\sum_{v\in V(G)\setminus (X\cup Y)}\deg(v)\le (k-1)|V(G)\setminus (X\cup Y)|$$ and thus, since $d(G)>k-1$ and hence $\sum_{v\in V(G)}\deg(v)> (k-1)|V(G)|$, we obtain
	\begin{equation*}\label{avg:X0Y02}(k-1)|X\cup Y|\ <\ \sum_{v\in X\cup Y}\deg(v)
		\ \le \ |X|(1-\sqrt{\eps})k+|Y|(1+\eps)k. \end{equation*}
	Therefore, 
	\begin{equation}\label{size:Y}|X| < 2\sqrt\eps|Y| < 3\sqrt\eps k.\end{equation}
	For each $v\in Y$ set $X_v:=N(v)\cap X$. Let $v^{\star}\in Y$ be a vertex that minimises $|X_v|$ among all $v\in Y$. So, 
	\begin{equation}\label{X_v}\text{for each }v\in Y, \hspace{.3cm}\deg(v,X)\ge |X_{v^\star}|.\end{equation}
	Let $T\in\mathcal T(k,\frac{\sqrt k}{1000})$. Now if $X_{v^\star}=\emptyset$, then the graph induced by $v^\star$ and a $k$-subset of $N(v^\star)$ satisfies the conditions of Lemma~\ref{lem:mindeg1}, with  $\nu:=\sqrt\eps$, and thus we can embed $T$. So, we will  from now on assume that $X_{v^\star}\not=\emptyset$.
		
	We use~Lemma~\ref{cut_off}, with  $\gamma:=168\sqrt\eps $,  to  obtain a subtree $(T^*,t^*)$ such that 
	\begin{equation}\label{mca}
	84\sqrt\eps k\le|T^*|\le 168\sqrt\eps k
	\end{equation}
	 and such that every component of $T-T^*$ is adjacent to $t^*$. We will now embed $T^*$ in a way that at least $|X_{v^\star}|$ vertices from $X$ will be used. Then, we embed the rest of $T$ into $G-X$ with the help of Lemma~\ref{bare}.
	 Before we start, we quickly prove two claims that will be helpful for the embedding of $T^*$.
	
	First, using~\eqref{size:Y} and the fact that $\delta(G)\ge \frac k2$, the following claim is easy to see. 
	
	\begin{claim}\label{paths} For every $x,x'\in V(G)$, there are more than $2^{-4}k$ internally disjoint paths of length at most $3$ connecting $x$ and $x'$.\end{claim}
	
	Second, we will see now that a useful subset of $Y$ can be `reserved' for later use.
	\begin{claim}\label{claim:leaves}There is a subset $Y'\subseteq Y\setminus\{v^\star\}$ of size at most $\lfloor5\sqrt{\eps}k\rfloor$ such that  all but at most $\lfloor2\eps k\rfloor$ vertices in $G-X$ have at least $|X|$ neighbours in $Y'$.\end{claim}
	To see this, suppose first that $|Y|\ge \lfloor 5\sqrt\eps k\rfloor+1$ and take any subset $Y'\subseteq Y\setminus\{v^\star \}$ of size $\lfloor 5\sqrt\eps k\rfloor$. Since every vertex $v$ in $G-X$ has degree at least $\lceil(1-\sqrt{\eps})k\rceil$ and since $n\le (1+\eps)k$, we know that $v$ has at least $\lceil3\sqrt\eps k\rceil\ge |X|$ neighbours in $Y'$, and we are done.  
	
	Assume now that $|Y|\le \lfloor 5\sqrt\eps k\rfloor$ and let us write $Z$ for the set of vertices in $G-X$ having less than $|X|$ neighbours in $Y\setminus\{v^\star \}$. Then one has the estimates 
	$$e(Y\setminus\{v^\star \},G)=\sum_{y\in Y\setminus\{v^\star \}}\deg(y)\ge (|Y|-1)k, $$
	and 
	$$e(Y\setminus\{v^\star \},G)=\sum_{z\in Z}\deg(z,Y\setminus\{v^\star \})+\sum_{z\not\in Z}\deg(z,Y\setminus\{v^\star \})\le |Z||X|+(n-|Z|)(|Y|-1).$$
	Therefore, as $|X|< 2\sqrt{\eps}|Y|$ by \eqref{size:Y}, and since by assumption $n\le(1+\eps)k$, we have $|Z|< 2\eps k$ and we can take $Y'=Y\setminus\{v^\star \}$. This finishes the proof of Claim~\ref{claim:leaves}.\\
	
	By applying Lemma~\ref{bare}, with $\ell =3$, we deduce that $T^*$ has either $|T^*|/12$ bare paths, each of length $3$, or it has at least $|T^*|/12$ leaves. The embedding of $T^*$ splits into two cases depending on the structure of $T^*$. 
	
	\begin{flushleft}\textbf{Case 1:} $T^*$ has  a set $\mathcal B$ of $|T^*|/12$ vertex disjoint bare paths, each of length $3$.\end{flushleft}
	
	We embed $T^*$ vertex by vertex in a pseudo-greedy fashion always avoiding $v^\star$. We start by embedding $t^*$ arbitrarily into any vertex of degree at least $(1-\sqrt\eps)k$ of $G-v^\star$. Now suppose we are about to embed a vertex $u'$ whose parent $u$ has already been embedded into a vertex $\phi(u)$. 
	If $u'$ is not the starting point of a path from $\mathcal B$ or if all of $X_{v^\star}$ is already used, we embed $u'$ greedily. Now assume that~$u'$ is the starting point of some $B\in\mathcal B$ and there is at least one unused vertex $x\in X_{v^\star}$. By Claim~\ref{paths} and since $|T^*|<2^{-4}k$, vertices $x$ and $\phi(u)$ are connected by a path~$P$ of length at most $3$ that uses only unoccupied vertices.  Embed $B$ (including $u$) into~$P$, and if $|B|>|P|$, choose its last vertices greedily.  Since by~\eqref{size:Y} and~\eqref{mca},
	$$|X|\le 3\sqrt{\eps} k< \frac{|T^*|}{12}=|\mathcal B|,$$
	we know that after embedding $T^*$ every vertex in $X_{v^\star}$ is used.

	\begin{flushleft}\textbf{Case 2:} $T^*$ has at least $|T^*|/12$ leaves.\end{flushleft}
	In this case, we cannot ensure that every vertex in $X_{v^\star}$ is used for the embedding of $T^*$, however, we can still guarantee that at least $|X_{v^\star}|$ vertices from $X$ are used.

	Because of our bound on the maximum degree of $T$, we can find a set $U^*\subseteq V(T^*)\setminus \{t^*\}$ of parents of leaves such that the number of leaves pending from $U^*$ is at least $6\sqrt\eps k$, which by~\eqref{size:Y} is greater than $2|X|$. We then take an independent set $U\subseteq U^*$ such that  for the set~$L$ of leaves pending from $U$ we have $|L|\ge |X|$, and such that $|U|\le |X|$.
	
	Starting from $t^*$ we embed $T^*$, following its natural order but leaving out the vertices from $L$. All vertices are embedded  greedily into $G-Y'$, except vertices from $U$ and their parents which are embedded in a different way. Assume $v\in V(T^*)$ is a parent of some vertex in $U$. Since $T^*$ is small, because of~\eqref{size:Y}, because of our assumption on the minimum degree of~$G$, and because of Claim~\ref{claim:leaves}, we may embed $v$ into a vertex having at least $|X|$ neighbours in~$Y'$. After this, we embed the children of $v$ in $U$ into unoccupied vertices of~$Y'$. Other children of $v$ are embedded greedily. At the end of this process we have embedded all of $T^*-L$. If we have used at least $|X_{v^\star}|$ vertices from $X$, we complete the embedding of $T^*$ greedily, so let us assume we have used less than $|X_{v^\star}|$ vertices from $X$. We embed the leaves pending from $U$ one by one into vertices from $X$ until we use $|X_{v^\star}|$ vertices, which is possible since $U$ was embedded into $Y'$ and because of~\eqref{X_v}. After this point, we simply embed the leftover leaves of $T^*$ greedily but always avoiding $v^\star$.\smallskip

	This finishes the case distinction.
Set $T':=T-(T^*-t^*)$. Denoting by $\phi$ the embedding we note that
\[|N(v^\star)\setminus (\phi(T^*)\cup X_{v^\star})|\ge k-|\phi(T^*)|-|X_{v^\star}|+|\phi(T^*)\cap X_{v^\star}|+|\phi(T^*)\setminus N(v^\star)|\ge |T'|-2.\]
Therefore, the graph $H$ induced by $v^\star$, $\phi(t^*)$ and any $(|T'|-2)-$subset of $|N(v^\star)\setminus (\phi(T^*)\cup X_{v^\star})|$ has order $|T'|$ and we may complete the embedding of $(T',t^*)$ by using Lemma~\ref{lem:mindeg1} for~$H$, with  $\nu:=86\sqrt\eps$,  fixing the image of $t^*$ as $\phi(t^*)$.
\end{proof}

\subsection{Using the regularity method}\label{sec:robust}

In this section we embed a given tree $T\in \mathcal T(k,\Delta)$ into $G$ using tools developed in~\cite{BPS1}. 
 The first auxiliary result that we need is stated as Proposition 5.1 and Remark~5.2 in~\cite{BPS1}. 
\begin{lemma}$\!\!${\rm\bf\cite{BPS1}}\label{emb:forest}
For all $\varepsilon\in (0,10^{-8})$ and $M_0,\Delta\in\NN$ there is $k_0$ such that for all $n,k_1, k_2\geq k_0$ 
	the following holds. 
	Let $G$ be an $n$-vertex graph having an $(\varepsilon,5\sqrt\eps)$-reduced graph $\mathscr R$ such that $|\mathscr R|\le M_0$ and $\mathscr R=(\mathscr A,\mathscr B)$ is connected and bipartite. If there is a subset $\mathscr V\subseteq \mathscr A$ such that
	\begin{enumerate}[(i)]
		\item $\deg(C)\ge (1+100\sqrt\eps)k_2 \cdot \frac{|\mathscr R|}n$ for all $C\in\mathscr V$; and
		\item $|\mathscr V|\ge (1+100\sqrt\eps)k_1 \cdot \frac{|\mathscr R|}n$,\end{enumerate}
then every tree $T\in\mathcal T(k,\Delta)$, with colour classes $A$ and $B$ obeying $|A|\le k_1$ and $|B|\le k_2$, can be embedded into $G$, with $A$ going to clusters in $\mathscr V$ and $B$ going to clusters in $\mathscr B$. 
\end{lemma}

Now we show that Lemma~\ref{emb:forest} is enough for embedding large trees in large graphs having a reduced graph which is connected and bipartite. 

\begin{lemma}\label{prop:robust_bip}
For all $\Delta\ge 2$, $M_0\in\mathbb N$, $\delta, \eps, \eta\in (0,1)$ with $\varepsilon\ll\eta\le\frac{\delta^2}{10^4}$ there is $k_0\in\mathbb{N}$ such that for all $k\ge k_0$, $n\in\mathbb N$ with $\delta^{-1}k\ge n\ge k$ the following holds.\\	
	Let $G$ be an $n$-vertex graph with an $(\varepsilon,\eta)$-regular partition and corresponding reduced graph~$\mathscr R$, with $|\mathscr R|\le M_0$, which is connected and bipartite with parts $\mathscr A$ and $\mathscr B$ such that $|\mathscr A|\ge |\mathscr B|$. If
	\begin{enumerate}[(i)]
		\item  $d(G)\ge(1-3\sqrt\eta)k$; \label{lagos1}
		\item $\delta(G)\ge(1-3\sqrt\eta)\tfrac{k}{2}$; and\label{lagos2}
		\item $|\bigcup\mathscr A|\ge (1+\delta)k$,\label{sucodetangerina}
	\end{enumerate}
	then $G$ contains every tree $T\in\mathcal T(k,\Delta)$.
\end{lemma}

\begin{proof} Given $\Delta,M_0$, $\eps$ and $\eta$, we choose $k_0$ as the output of Lemma~\ref{emb:forest}. Given $G$  as in Lemma~\ref{prop:robust_bip}, we suppose for contradiction  that some $T\in\mathcal T(k,\Delta)$ cannot be embedded into~$G$. Set $$t=\frac{|\mathscr R|}{n}$$ and let $|\bigcup \mathscr A|=a$ and $|\bigcup\mathscr B|=b$. 
We claim that 
\begin{equation}\label{bbb}
b\ge \Big(1+\frac{\delta}4\Big)\frac{k}{2}.
\end{equation}
Indeed, otherwise can use~\eqref{lagos1} to calculate that 
	\begin{align*}
	 (1-3\sqrt\eta)kn\le 2e(G)\le 2ab
	 \le\Big(1+\frac{\delta}4\Big)ka\le\Big(1+\frac{\delta}4\Big)k\cdot \Big(1-\frac\delta{4}\Big)n\le \Big(1-\frac{\delta^2}{16}\Big)kn
	 \end{align*}
	where the second to last inequality follows from the fact that because of~\eqref{lagos2} we have $a=n-b\le n- (1-3\sqrt\eta)\frac k2\le (1-\frac\delta{4})n$. But this is a contradiction to our assumptions on $\eta$ and~$\delta$. This proves~\eqref{bbb},
		and so, we also know that  	
\begin{equation}\label{offthegrid}
		|\mathscr A|\ge|\mathscr B|\ge \Big(1+\frac\delta4\Big)\frac k2t. 
\end{equation}		
		Now we turn to the tree $T$. Let $A$ and $B$ denote its colour classes, and assume $|A|\ge |B|$. Moreover, we may assume that  	
\begin{equation}\label{bip:sizes}(1-4\sqrt\eta)\frac k2<|B|\le \frac{k+1}2\hspace{.5cm}\text{and}\hspace{.5cm}\frac{k+1}2\le |A|\le (1+4\sqrt\eta)\frac k2.\end{equation}
as otherwise,  since $\eps\ll\eta$ we have $\delta(G)\ge (1+100\sqrt\eps)|B|$ and so, by~\eqref{sucodetangerina}, we can use Lemma~\ref{emb:forest} to embed $T$. 

Let $\mathscr V_A\subseteq \mathscr A$ and $\mathscr V_B\subseteq \mathscr B$ be the sets of all clusters of degree at least $(1+\sqrt\eta)\frac k2t$. We claim that 
\begin{equation}\label{VAVB}
|\mathscr V_A|+|\mathscr V_B|\ge (1+\sqrt\eta)kt.
\end{equation}
 Suppose this is not the case. Then Fact~\ref{fact:3}~\eqref{fact:3,1}, condition~\eqref{lagos1}, and~\eqref{offthegrid} imply that 
\begin{eqnarray*}(1-3\sqrt\eta)kt|\mathscr R| &\le& 2e(\mathscr R)\\ &\le& |\mathscr V_A||\mathscr B|+|\mathscr V_B||\mathscr A|+(1+\sqrt\eta)\frac k2t\Big(|\mathscr R|-|\mathscr V_A|-|\mathscr V_B|\Big)\\
&=&(1+\sqrt\eta)\frac k2t|\mathscr R|+|\mathscr V_A|\Big(|\mathscr B|-(1+\sqrt\eta)\frac k2t\Big)+|\mathscr V_B|\Big(|\mathscr A|-(1+\sqrt\eta)\frac k2t\Big)\\
&<&(1+\sqrt\eta)\frac k2t|\mathscr R|+ (1+\sqrt\eta)kt\cdot \Big(\frac\delta8-\sqrt\eta\Big)\frac k2t.\end{eqnarray*}
 Therefore, and since $n\ge k$, we have
 $$\frac 12t\cdot k\le (1-7\sqrt\eta)tn =(1-7\sqrt\eta)|\mathscr R| <  (1+\sqrt\eta)\Big(\frac\delta4-\sqrt\eta\Big)kt\le \frac 32\cdot \frac\delta4 kt,$$
 a contradiction. So, assuming that $|\mathscr V_A|\ge (1+\sqrt\eta)\frac k2 t$, by Lemma~\ref{emb:forest} we can embed $T$ into~$G$, with $A$ going to clusters in $\mathscr V_A$ and $B$ going to clusters in $\mathscr B$.
 \end{proof}

Now we turn to the case when the reduced graph is connected, non-bipartite and large. 

\begin{lemma}[\cite{BPS1}, Proposition 5.8]\label{lem:nonbipcon}
For all $\Delta\ge 2$ and $\eps\in(0,10^{-8})$, there is $k_0\in\mathbb N$ such that for all $n,k\ge k_0$ and for every $n$-vertex graph $G$ the following holds. If $G$  has an $(\eps,5\sqrt\eps)$-regular partition, and the corresponding reduced graph $\mathscr R$ has a non-bipartite connected component~$\mathscr C$ which contains a matching with at least $(1+100\sqrt\eps)\frac k2 \frac{|\mathscr R|}{n}$ edges, then~$G$ contains every tree $T\in\mathcal T(k,\Delta)$. 
\end{lemma}

With the help of Lemma~\ref{lem:nonbipcon} we can derive some useful information on the structure of the reduced graph of~$G$ if it is connected and non-bipartite, and $G$ fails to contain a copy of some tree $T\in \mathcal T(k,\Delta)$. 

\begin{lemma}\label{lem:str_nonbip}
For all $\Delta\ge 2$, $M_0\in \mathbb N$,  $\delta, \eps, \eta\in(0,1)$ with $\eps\ll \eta\le\frac{\delta^4}{10^{8}}$, there is $k_0\in\mathbb{N}$ such that for all $k,n\ge k_0$ with $\delta^{-1}k\ge n\ge (1+\delta)k$ the following holds. Let $G$ be an $n$-vertex graph that admits an $(\varepsilon,\eta)$-regular partition into $M_0$ parts, and assume the corresponding $(\varepsilon,\eta)$-reduced graph is connected and non-bipartite. If furthermore,
	\begin{enumerate}[(i)]
		\item  $d(G)\ge(1-3\sqrt\eta)k$; and
		\item $\delta(G)\ge(1-3\sqrt\eta)\tfrac{k}{2}$,
	\end{enumerate}
and there is a $T\in\mathcal T(k,\Delta)$ that cannot be embedded into $G$, then $G$ has a subgraph $G'\subseteq G$ of size $|G'|\ge |G|-M_0$ such that there is a  partition $V(G')=I\cup V_1\cup V_2$ with
	\begin{enumerate}[(a)]
		\item\label{str:Vismall} $|V_i|=(1\pm3\sqrt\eta)\frac{k}{2}$ for $i=1,2$;
		\item\label{str:I}\label{str:edges} $I$ is an independent set in $G'$ and there are no edges between $I$ and $V_2$ in $G'$;
		\item\label{str:deg1} $\deg_{G'}(x)\ge (1-5\sqrt[4]\eta)n$ for at least $(1-4\sqrt[4]\eta)|V_1|$ vertices $x\in V_1$;
		\item\label{str:deg2} $\deg_{G'}(y)\ge (1-3\sqrt[8]\eta)k$ for at least $(1-2\sqrt[8]\eta)|V_2|$ vertices $y\in V_2$.
	\end{enumerate}
\end{lemma}
\begin{proof}
	Let $k_0\ge \frac{M_0}{\eps}$ be at least as large as the output of Lemma~\ref{lem:nonbipcon} for $\frac\eps5$ and $\Delta$. Applying Lemma~\ref{lem:matching2} to $G$, with $\ell=M_0$ and $t=(1-3\sqrt\eta)\frac k2$, we find a subgraph $G'$ of size $|G'|\ge n-M_0$  that admits an $(5\eps,\frac \eta 2)$-regular partition. Moreover, the  corresponding reduced graph $\mathscr R$ contains a matching $\mathscr M$ and a disjoint independent set $\mathscr I$ such that $V(\mathscr R)=\mathscr I\cup V(\mathscr M)=\mathscr I\cup\mathscr V_1\cup \mathscr V_2$ and $N_{\mathscr R}(\mathscr I)\subseteq \mathscr V_1$. 
	
	Letting $I=\bigcup \mathscr I$ and $V_i=\bigcup\mathscr V_i$ for $i=1,2$ we have~\eqref{str:I}. Furthermore, because of Lemma~\ref{lem:nonbipcon} we know that $|\mathscr V_i|\le(1+\eta)\frac k2\tfrac{|\mathscr R|}{n}$ and thus $| V_i|\le (1+\eta)\tfrac k2$ for $i=1,2$. Therefore, and because of condition~(ii) we have~\eqref{str:Vismall}. 
	
	In order to see~\eqref{str:deg1} and~\eqref{str:deg2}, we do the following. For any subset $A\subseteq V(G')$ let  $d_{A}$ denote the average degree in $G'$ of the vertices in $A$.  By~\eqref{str:edges}, we have  $d_{I}\le|V_1|\le (1+\eta)\frac k2$. By condition~(i) and since $\deg_{G'}(x)\ge \deg_G(x)-M_0$ for every $x\in V(G')$, we have
	\begin{eqnarray*}(1-4\sqrt\eta)kn\ \le \ 2e(G')&=&
		|I|d_I+|V_1|d_{V_1}+|V_2|d_{V_2}\\
		&\le&\textstyle(1+\eta)\frac k2 ( |I|+d_{V_1}+d_{V_2})\\
		&\le&\textstyle(1+\eta)\frac k2(n-(|V_1|+|V_2|) +d_{V_1}+d_{V_2})\\
		&\le&\textstyle(1+\eta)\frac k2(n-(1-3\sqrt\eta)k +d_{V_1}+d_{V_2}),
		\end{eqnarray*}
and therefore,
	\begin{equation}\label{eq:degV1V2}
	d_{V_1}+d_{V_2}\ge (1-8\sqrt\eta)n +(1-3\sqrt\eta)k. 
	\end{equation}
Because of~\eqref{str:edges}, we have $d_{V_2}\le |V_1|+|V_2|\le  (1+\eta)k$. Thus~\eqref{eq:degV1V2} implies that $d_{V_1}\ge (1-12\sqrt\eta)n$. Since $d_{V_1}\le n$ and  since $n\le \delta^{-1}k$, inequality~\eqref{eq:degV1V2} also implies that $d_{V_2}\ge (1-\sqrt[4]\eta)k$. Apply Lemma~\ref{lem:con2} to $f_1(v)=\deg_{G'}(v)$ for $v\in V_1$, with parameters $t_1=(1-12\sqrt\eta)n$, and $\eps_1=16\sqrt\eta$,  and to $f_2(v)=\deg_{G'}(v)$ for $v\in V_2$,  with  $t_2=(1-\sqrt[4]\eta)k$ and $\eps_2=4\sqrt[4]\eta$,  to obtain~\eqref{str:deg1} and~\eqref{str:deg2}. \end{proof}

The next lemma finishes the analysis of the non-bipartite case.

\begin{lemma}\label{prop:robust2}
For all $\Delta\ge 2, M_0\in\mathbb N$, $\delta, \eps, \eta\in (0,1)$ with $\eps\ll \eta\le\frac{\delta^{8}}{10^{80}}$, there is $k_0\in\mathbb{N}$ such that for all $k,n \ge k_0$ with $\delta^{-1}k\ge n\ge (1+\delta)k$ the following holds. Let $G$ be an $n$-vertex graph that admits an $(\varepsilon,\eta)$-regular partition into at most $M_0$ parts and assume the corresponding reduced graph is connected and non-bipartite.
If
	\begin{enumerate}[(i)]
		\item  $d(G)\ge(1-3\sqrt\eta)k$; and
		\item $\delta(G)\ge(1-3\sqrt\eta)\tfrac{k}{2}$,
	\end{enumerate}
	then $G$ contains every tree $T\in\mathcal T(k,\Delta)$.
\end{lemma}
\begin{proof} Let $k_0$ be the output of Lemma~\ref{lem:str_nonbip} and let $G$ and  $T\in\mathcal T(k,\Delta)$ be given. If we cannot embed $T$ into $G$, then by Lemma~\ref{lem:str_nonbip} we find a subgraph $G'\subseteq G$ and a partition $V(G')=I\cup V_1\cup V_2$ fulfilling the  properties of Lemma~\ref{lem:str_nonbip}.

 Let $U_1\subseteq V_1$ be the set of all vertices $x\in V_1$ with $\deg_{G'}(x)\ge (1-5\sqrt[4]\eta)n$, and let $U_2\subseteq V_2$ be the set of all vertices $x\in V_2$ with $\deg_{G'}(x)\ge (1-3\sqrt[8]\eta)k$.  In particular, because of Lemma~\ref{lem:str_nonbip}~\eqref{str:Vismall}, we have that
	\begin{equation}\label{deg:I}\textstyle\text{each vertex $x\in U_1$ has at least $(1-\sqrt\eta)|I|$ neighbours in $I$}.\end{equation}
	Also, note that $|U_1|\ge (1-4\sqrt[4]\eta)|V_1|\ge |V_1|-\sqrt[8]\eta k$ and $|U_2|\ge (1-2\sqrt[8]\eta)|V_2|$,
	by Lemma~\ref{lem:str_nonbip}~\eqref{str:Vismall}, ~\eqref{str:deg1} and~\eqref{str:deg2}.
	Let $H$ be the graph induced by $U_1$ and $U_2$. Note that because of Lemma~\ref{lem:str_nonbip}~\eqref{str:edges} and~\eqref{str:deg2}, 
	we know that the vertices from $U_2$ have minimum degree at least $(1-6\sqrt[8]\eta)k$ in $H$, and 
	because of Lemma~\ref{lem:str_nonbip}~\eqref{str:Vismall} and~\eqref{str:deg2}, 
the vertices from $U_1$ have minimum degree at least $(1-9\sqrt[4]{\eta})n-2\sqrt[8]{\eta}k-|I|\ge (1-3\sqrt[8]\eta)k$ in $H$.
Hence, 
\begin{equation}\label{mindegH__}
\delta(H)\ge (1-6\sqrt[8]\eta)k.
\end{equation}
 So, by Lemma~\ref{lem:greedy} every tree with at most $(1-6\sqrt[8]\eta)k$ edges can be embedded greedily into~$H$. Let $(T^*,t^*)$ be the subtree given by Lemma~\ref{cut_off} for $\gamma=\frac 12$, so that $\frac k4\le|T^\star|\le\frac k2$ and every component of $T-T^*$ is adjacent to $t^*$. We apply Lemma~\ref{bare} to $T^*$, with $\ell=3$, which splits the proofs into two cases.
		\begin{flushleft}\textbf{Case 1:} $T^\star$ has a set $\mathcal B$ of $|T^\star|/12$ vertex disjoint bare paths, each of length $3$.\end{flushleft}
		Note that each vertex from $H$ has at least $\Delta$ neighbours in $U_1$, because of Lemma~\ref{lem:str_nonbip}~\eqref{str:Vismall} and our bound from~\eqref{mindegH__}, which will be tacitly used in what follows.
		
		We embed $t^\star$ into any vertex from $H$. The rest of $T^\star$ will be embedded in DFS order into~$H$. We will use the following strategy until we have occupied  $\lceil\frac\delta{100}k\rceil$ vertices from~$I$. 
		For each path $P\in \mathcal B$, we proceed as follows. We embed the first vertex $v_1$ of the path $P$  into a vertex $u_1\in U_1$, and then find another vertex $u_3\in U_1$ which has a common neighbour $u_2$  with $u_1$ in $I$. Note that the vertex $u_3$ exists because of~\eqref{deg:I}. We then embed the middle vertex $v_2$ of $P$  into $u_2\in I$, and the end point $v_3$ into $u_3\in U_1$. The remaining vertices of $T^\star$ are embedded greedily into $H$.
	\begin{flushleft}\textbf{Case 2:} $T^\star$ has $|T^\star|/12$ leaves.\end{flushleft}
	In this case, the embedding of $T^\star$ follows a similar strategy. We embed $t^\star$ into any vertex from $H$ and the rest will be embedded in DFS order. 
	We take care to embed all parents of leaves into $U_1$ and all leaves into $I$, until we have used  $\lceil\tfrac\delta{100}k\rceil$ vertices from~$I$. The remaining vertices of $T^\star$ are embedded greedily into $H$.\\
	
	Now, let $m$ be the number of vertices we have embedded so far into $H$, and let $H'\subseteq H$ contain all unused vertices of $H$. By our embedding strategy, we have that $m \le |T^\star|-\tfrac\delta{100}k$.
	Therefore, and by~\eqref{mindegH__}, 	\[\delta(H')\ge (1-6\sqrt[8]\eta)k-m\ge (1-6\sqrt[8]\eta)k+\tfrac\delta{100}k-|T^\star|\ge(1+\tfrac\delta{200})k-|T^\star|, \] and so we can finish the embedding of $T$ by embedding $T-T^*$ greedily into $H'$. 
\end{proof}

\subsection{Proof of Theorem~\ref{thm:main}}\label{sec:proof_main}

In this subsection we prove Theorem~\ref{thm:main} with the help of the results from the previous  subsections.
In order to do this, we need a result that follows from Theorem~1.9 in~\cite{BPS1} (the original Theorem~1.9 allows for a weaker bound on the maximum degree of $T$).   

\begin{lemma}$\!\!${\rm\bf\cite{BPS1}}\label{ES:app}
For all $\Delta\ge2$ and $\delta,\theta\in(0,1)$ there is $n_0\in\mathbb N$ such that for all $n\ge n_0$ and $k\in\mathbb N$ with $n>k\ge \delta n$ the following holds. Let $G$ be an $n$-vertex graph with $d(G)\ge (1+\theta)k$, then $G$ contains every tree $T\in\mathcal T(k,\Delta)$. 
	\end{lemma} 
	

Now we are ready for the proof of Theorem~\ref{thm:main}.
\begin{proof}[Proof of Theorem~\ref{thm:main}]
	Given $\Delta$ and $\delta$, we set $\nu=\min\{\frac{\delta^2}{2^{10}},\frac{1} {10^{11}},\frac{1}{25\Delta^2}\}$ and we fix  parameters $\eps$, $\eta$, $\theta$ such that
	\[0<\varepsilon\ll\eta\ll\theta\le\frac{\nu^{8}}{10^{80}}.\]
Let $k_0$ be the maximum of $\frac{3}{\eps}$ and the outputs of Lemma~\ref{reg:deg}, Lemma~\ref{prop:robust_bip}, Lemma~\ref{prop:robust2} and Lemma~\ref{ES:app} (with $\nu$ playing the role of $\delta$, and $m_0=\lceil\frac 1\eps\rceil$). Set $n_0=\lceil\delta^{-1}k_0\rceil$.

By Proposition~\ref{prop:small} we may assume that $|G|\ge (1+\nu)k$ and if $G$ is $\nu$-bipartite, Proposition~\ref{prop:small:bip} allows us to  assume that the larger bipartition class of $G$ has  at least $(1+\nu)k$ vertices. Now the regularity lemma (Lemma~\ref{reg:deg}) provides us with a subgraph $G'$ with $|G'|\ge (1-\eps)n$  that has an $(\eps,\eta)$-regular partition. Let $\mathscr R$ be the corresponding reduced graph and let $\mathscr{U}_1, \ldots , \mathscr U_\ell$ be the connected components of $\mathscr R$. Then, since we may assume that $\delta(G)\ge\frac k2$ (see the footnote in the Introduction), we have
\[
\text{$\deg_{G'}(x)\ge (1-2\sqrt\eta)\deg_{G}(x)\ge (1-2\sqrt\eta)\frac k2$ \ for all  $x\in V(G')$,}
\]
and therefore
\[ \frac{\ell k}4\le(1-2\sqrt\eta)\tfrac{k}{2}\ell\le \sum_{i\in[\ell]} \textstyle{|\bigcup\mathscr U_i|}\le n\le \delta^{-1}k,\]
implying that
\begin{equation}\label{ellanddelta}
\ell\le 4\delta^{-1}.
\end{equation}
 We set $U'_i=\bigcup \mathscr U_i$ for each $i\in[\ell]$.
	
\begin{claim}\label{prop:stability}Suppose that exists $T\in\mathcal T(k,\Delta)$ which cannot be embedded into $G'$, then

	\begin{enumerate}[(i)]
		\item $d(G'[U'_i])= (1\pm\frac\nu2)k$ and $\delta(G'[U'_i])\ge (1-\frac\nu2)\frac k2$ for all $i\in [\ell]$; and\label{iiiiiii}
		\item\label{stb:sizes} for each $i\in[\ell]$ either 
		\begin{enumerate}[(a)]
			\item $G'[U'_i]$ is non-bipartite and $|U_i|=(1\pm\frac\nu2)k$, or
			\item $G'[U'_i]$ is bipartite with $V(U_i)=A_i\cup B_i$ such that $|A_i|,|B_i|=(1\pm\frac\nu2)k$.\end{enumerate}
	\end{enumerate}
\end{claim}

In order to see this claim, observe that  since $T$ cannot be embedded into $G'$, Lemma~\ref{ES:app} implies that $d(G'[U'_i])<(1+\theta)k$ for each $i\in[\ell]$. Note that 
	\[\sum_{i=1}^\ell\frac{|U'_i|}{n}d(G'[U'_i])=d(G')\ge (1-3\sqrt\eta)k.\]
	Set $t=(1-3\sqrt\eta)k$. Applying Lemma~\ref{lem:con} with $N=\ell$, $\mu(i)=|U'_i|/n$ and $f(i)=d(G'[U'_i])$, and with $\sqrt\theta$ in the role of $\eps$, we see that the set $I=\{i\in[\ell]:d(G'[U'_i])<(1-2\sqrt\theta)t\}$ satisfies \[\frac{t|I|}{2n}\le\mu(I)\le 2\sqrt\theta\]
	(where for the first inequality we use that $|U_i|>\frac t2$ for each $i$). Thus, $|I|\le 8\delta^{-1}\sqrt\theta<1$. In other words, $I=\emptyset$, and therefore, for each $i\in[\ell]$ we have
    \begin{equation}\label{stb:avg}
    d(G'[U'_i])\ge (1-2\sqrt\theta)(1-3\sqrt\eta)k\ge (1-3\sqrt\theta)k.
    \end{equation} 
    This, together with the minimum degree in $G'$, proves~\eqref{iiiiiii}. In order to see~\eqref{stb:sizes}, we use~\eqref{stb:avg} and  Lemmas~\ref{prop:robust_bip} and~\ref{prop:robust2}. This proves Claim~\ref{prop:stability}.\\

 Now we distribute the vertices from $G-G'$ into the sets~$U'_i$. We successively assign each leftover vertex to the set $U'_i$  it sends most edges to (or to any one of these sets, if there is more than one). Then for each $i\in[\ell]$ and all $x\in U_i$ we have
 \begin{equation*}
 \deg(x,U'_i)\ge \frac{k}{2\ell}\ge \frac\delta8k,
 \end{equation*}
where we used~\eqref{ellanddelta} for the second inequality. Since we add at most $\eps n\ll\nu k$ vertices to each set, we end up with a partition  $V(G)=U_1\cup\ldots \cup U_\ell$ satisfying, for each $i\in[\ell]$,
\begin{enumerate}[(I)]
		\item $d(G[U_i])= (1\pm\nu)k$ and $\delta(G[U_i])\ge \frac\delta 8k$;  \label{avg:degree}
		\item  $\deg(x,U_i)< (1-\nu)\frac k2$ for less than $\nu k$ vertices $x\in U_i$; and
		\item either  $G[U_i]$ is non-bipartite and $|U_i|= (1\pm\nu)k$, or
			 $G[U_i]$ is $\nu$-bipartite with $U_i=A_i\cup B_i$ such that  $|A_i|,|B_i|=(1\pm\nu)k$.\label{uuuuuiiiiii}
	\end{enumerate}
For each $i\in[\ell]$, we use Lemma~\ref{lem:con2} for $f(x)=\deg(x,U_i)$, with $2\nu$ playing the role of $\eps$, to deduce that 
\begin{equation}\label{ES:degree}\deg(x,U_i)\ge (1-\sqrt{2\nu})k\text{ for at least $(1-\sqrt{ 2\nu})|U_i|$ vertices from $U_i$.}
\end{equation}

Now we embed $T$ using this structural information of $G$. We apply Lemma~\ref{cut_off} to $T$, with $\gamma=\frac 12$, to obtain a subtree $(T,t^*)$ with $\frac k4\le |T^*|\le\frac k2$  such that every component of $T-T^*$ is adjacent to $t^*$. Moreover, since $\Delta(T)\le \Delta$ there is a component $T'$ of $T-T^*$ with $\frac{k}{2\Delta}\le |T'|\le \frac{3k}4$. 

Note that if there are no edges between different sets $U_i$, then an averaging argument shows that there is  $i^\star\in[\ell]$ such that $d(G[U_{i^\star}])\ge d(G)>k-1$. But then, because of~(III) and because of Theorem~\ref{smalltheorem}, we are done. Thus, we may assume that there is an edge $u_iu_j$ with $u_i\in U_i$ and $u_j\in U_j$.
We map $t^*$ into  $u_i$ and map the root of $T'$ into  $u_j$. Note that by~\eqref{avg:degree}, we have
\begin{equation}\label{ES:mindegree}\delta(G[U_i])\ge \frac\delta8k\ge 4\sqrt\nu k\ge \sqrt{2\nu}|U_i|+\Delta\end{equation}
and that~\eqref{uuuuuiiiiii}, together with our choice of $\nu$ ensures that $\sqrt{2\nu}|U_i|\le \frac{k}{2\Delta}$. So,
 we may finish the proof by using  Lemma~\ref{lem:greedy} and Lemma~\ref{lem:greedy3} to embed $T-T'$ into $U_i$ and $T'$ into $U_j$, which we can do because of~\eqref{ES:degree} and~\eqref{ES:mindegree}.  

\end{proof}

\section{Trees with up to linearly bounded maximum degree}\label{sec:linear}

\subsection{Proof of Theorem~\ref{thm:es_app}}

We will need the
 following lemma, which will be proved in Section~\ref{sec:pow}.

\begin{lemma}\label{prop:connected_linear}
	For all $\delta\in (0,\frac 12)$ there are $k_0\in\NN$ and $\rho\in(0,1)$ such that for all $k\ge k_0$ and every $n$-vertex graph $G$ with $n\ge k\ge \delta n$  the following holds. If $\delta(G)\ge(1+\delta)\frac{k}{2}$ and 
at least $\lceil\delta n\rceil$ vertices of $G$ have degree at least $(1+\delta)k$,
	then $G$ contains every tree $T\in\mathcal T(k,\rho k)$. 
\end{lemma}

Now, Theorem~\ref{thm:es_app} follows from the Lemma~\ref{prop:connected_linear} together with
 Lemma~\ref{lem:con} from the appendix.

\begin{proof}[Proof of Theorem~\ref{thm:es_app}]
 Given  $\delta$ from Theorem~\ref{thm:es_app} (note that we may assume~$\delta<\frac 12$), let $k_0$ and $\rho$ be the output of Lemma~\ref{prop:connected_linear} for input $\delta^{12}$. Set $n_0:=\delta^{-1}k_0$ and set $\gamma:=\rho$. 
	
	Given $k$, $n$ and $G$, a standard argument\footnote{This is the same argument as the one given in the footnote in the Introduction, replacing $k$ with $(1+\delta )k$.} gives a subgraph $G'$  with $d(G')\ge (1+{\delta})k$ and $\delta(G')\ge (1+\delta)\frac k2$.
	If there are $\lceil\delta^{12} |G'|\rceil$ vertices in $G'$ of degree at least $(1+\delta^{12})k$, we are done by Lemma~\ref{prop:connected_linear}.
	So assume otherwise. Then by Lemma~\ref{lem:con}, with $f(v)=d_{G'}(v)$, $\eps=\delta^{12}$ and $t=(1+\delta)k$, we know $G'$ has at most $\delta^3|G'|\le \delta^3 n\le \delta^2 k$ vertices of degree less than $(1-\delta^3)(1+\delta)k$. Since $(1-\delta^3)(1+\delta)k\ge (1+\frac \delta 2)k$, we can simply delete these vertices, obtaining a subgraph $G''$ of $G$ with $\delta(G')\ge k$. We  greedily embed $T$ into $G''$.
\end{proof}

\subsection{Preparing for the proof of Lemma~\ref{prop:connected_linear}}

We start by stating a standard tool  (see~\cite{AKS, BPS1, LKS4, RS19a} for other versions and a proof).

\begin{lemma}\label{prop:cut2} Let $\beta\in(0,1)$. If $(T, r)$ is a rooted tree with $k\ge \beta^{-1}$ edges, then there is a set $S\subseteq V(T)$ with $r\in S$ and $|S|\le \beta^{-1} +2$ such that $|P|\le\beta k$ for each component $P$ of $T-S$.  
\end{lemma}

We now show a variant of Lemma~\ref{prop:cut2}.
\begin{lemma}\label{prop:cut_linear}
	For all $\beta\in(0,\frac 12)$  and for every tree $(T,r)$  with $k\ge \beta^{-1}$ edges and $\Delta (T)\le \frac{\beta^2}{2} k$  there is a set $S\subseteq V(T)$ with $r\in S$ and $|S| < \beta k$ such that each $s\in S$ is at even distance from $r$ and each component of $T-S$ has at most $\beta k$ vertices. 
\end{lemma}
\begin{proof}
	Given  $(T,r)$, Lemma~\ref{prop:cut2} yields a set  $S'$ with 
$ |S'|\le\frac 1\beta +2< \frac{2}{\beta}.$
 Let $S_{odd}$ be the set of all vertices in $S'$ that lie at odd  distance from $r$, and set  $S:=(S'-S_{odd})\cup N_T(S_{odd}).$	
Note that each component of $T-S$ either is a component of $T-S'$, or consists of a single vertex from~$S_{odd}$. To see that $|S| < \beta k$, note that	
	$\textstyle |S|\le |S'|+|S_{odd}|\cdot\Delta(T)< \frac{2}{\beta}\cdot \frac{\beta^2}{2}  k= \beta k.$
\end{proof}

 The next lemma will help us with grouping the components fo $T-S$ into convenient sets.

\begin{lemma}\label{lem:num_balance}
	Let $I$ be a finite set, let $M, \lambda>0$, and let $a_i, b_i\in \mathbb R$, with $a_i+b_i\leq \lambda$, for each $i\in I$. Then  there is a set $J\subseteq I$ such that 
	$$\textstyle\min\{M-\lambda,\sum_{i\in I}(a_i+b_i)\}\leq\sum_{i\in J}(a_i+b_i)\leq M \ \ \ \text{and} \ \  \ \frac{\sum_{i\in J}a_i}{\sum_{i\in J}b_i}\geq \frac{\sum_{i\in I}a_i}{\sum_{i\in I}b_i}.$$
\end{lemma}

\begin{proof}
	 Define a total order $\preceq$ on $I$ by setting $i\preceq j$ if $\frac{b_i}{a_i}<\frac{b_j}{a_j}$, and ordering arbitrarily those $i,j$ with $\frac{b_i}{a_i}=\frac{b_j}{a_j}$. Let $j^{*}$ be maximal with $\textstyle\sum_{i\preceq j^{*}}(a_i+b_i)\leq M$ and set $J:=\{j\in I\,:\,j\preceq j^{*}\}$. It is is easy to see that this choice is as desired.
\end{proof}

We will now find  a specific structure in the regularised host graph $G$. 

\begin{lemma}\label{lem:structure}
For all $\varepsilon, \eta, \delta>0$ with $\varepsilon<\eta<\delta<1$ and  $M_0\ge \frac 1{\eps}$ there is $k_0\in\mathbb N$ such that for all $n,k\ge k_0$ and for every $n$-vertex graph $G$ with $\delta(G)\ge(1+{\delta})\frac{k}{2}$ and
  at least $\lceil\delta n\rceil$ vertices of degree at least $(1+\delta)k$
  the following holds. \\
If $G$ has an $(\varepsilon,\eta)$-regular partition into $M_0$ parts,
	then $G$ has a  subgraph $G'$ on $n'\ge n-M_0$ vertices that has a $(5\varepsilon,\eta-\varepsilon)$-regular partition with $2M_0$ parts. Moreover, the  corresponding reduced graph $\mathscr R'$ contains two matchings $\mathscr M_W$ and $\mathscr M_V$, a bipartite subgraph $\mathscr {H=(A,B)}\subseteq \mathscr R'\setminus V(\mathscr M_W)$, and a cluster $X\in V(\mathscr R')$ satisfying
	\begin{enumerate}[(I)]
		\item\label{lem:str:I} $V(\mathscr M_W)\cap V(\mathscr M_V)=\emptyset=\mathscr{A}\cap V(\mathscr M_V)$;
		\item\label{lem:str:II}\label{lem:str:III}\label{lem:str:IV}  $V(\mathscr M_W)\cup \mathscr A\subseteq N(X)$,
		 and every edge in $\mathscr M_V$ has  exactly one endpoint in $N(X)$;
		\item\label{lem:str:VI}  $|V(\mathscr M_W)|+\frac{|V(\mathscr M_V)|}{2}+|\mathscr{A}| \ge (1+\tfrac\delta2)k\tfrac{|\mathscr R'|}{n'}$; and
		\item\label{lem:str:V} $\deg_{\mathscr R'}(A,\mathscr{B})\ge (1+\tfrac{\delta}2)\tfrac k2\tfrac{|\mathscr R'|}{n'} - \tfrac{|V(\mathscr M_W)|}{2}\text{ for every $A\in \mathscr A$.}$
	\end{enumerate}	
\end{lemma}
\begin{figure}[h!]
	\centering
	\includegraphics[width=0.45\linewidth]{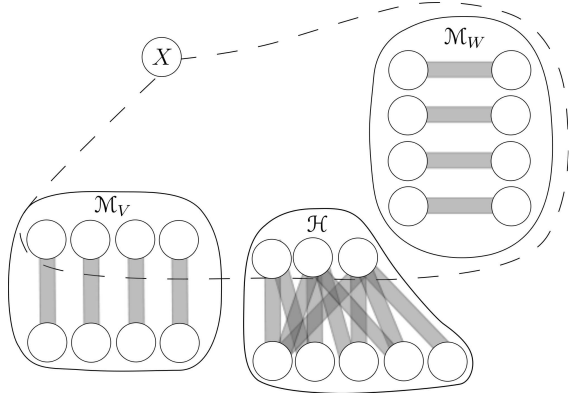}
	\caption{Structure given by Lemma~\ref{lem:structure}}
	\label{lab:structure}
\end{figure}

\begin{proof}
Set $k_0=\frac{M_0}\eps$.	Apply Lemma \ref{lem:matching2} to $G$, with $t=(1+\delta)\frac k2$ and $\ell=|\mathscr R|$, to obtain a subgraph~$G'$, with a $(5\varepsilon,\eta-\varepsilon)$-regular partition into $2M_0$ parts whose  corresponding reduced graph $\mathscr R'$  contains a matching $\mathscr M$ and an independent set $\mathscr I$ with the properties stated in the lemma. 
By the choice of $k_0$ and $\eps$, and by our assumption on $G$,
at least $\lceil\tfrac{\delta}2|G'|\rceil$ vertices of $G'$ have degree at least $(1+\tfrac\delta 2)k$.
	So,  there is a cluster $X\in V(\mathscr R')$ with 
	\begin{equation}\label{deg_C_S}
	\deg_{\mathscr R'}(X)\ge (1+\tfrac\delta2){k}\tfrac{|\mathscr R'|}{|G'|}.
	\end{equation}
	
	Let $\mathscr M_W$ be a maximal matching contained in $N(X)$, so that for every $CD\in \mathscr M_W$ either $CD\in \mathscr M$ or $C\in\mathscr  I$ and $D\in V(\mathscr{M})$. This choice ensures that there are no edges between 
	$\mathscr{A}:=\big(N_{\mathscr R'}(X)\cap\mathscr{I}\big)\setminus V(\mathscr M_W)$
	 and $\big(N_{\mathscr R'}(X)\cap V(\mathscr{M})\big)\setminus V(\mathscr M_W)$. 
	Set  
	$\mathscr{B}:=N_{\mathscr R'}(\mathscr{A})\setminus V(\mathscr M_W).$
	Let $\mathscr M_V$ consist of all edges in $\mathscr M-\mathscr M_W$ having one endpoint in $N(X)$. 
	
	By construction, properties $(I)-(II)$ hold, and $(III)$ holds because of~\eqref{deg_C_S}. Finally,~$(IV)$ holds because of our assumption on the minimum degree of $G$, and since any $A\in\mathscr A \subseteq \mathscr I$ sees at most one endpoint of each edge from $\mathscr M_W$.
\end{proof}

\subsection{Proof of Lemma~\ref{prop:connected_linear}}\label{sec:pow}

\begin{proof}
Given $\delta$, we choose $\eps$ and $\eta$ such that 
	$0<\eps\ll\eta\ll\delta.$ Apply
Lemma~\ref{reg:deg} with parameters  $\frac\eps 5$ and $m_0=\frac 1\eps$ to obtain numbers $N_0$ and $M_0$. Set $k_0:=\max\{N_0, \frac{2M_0}{\eps}, k'_0\}$ where $k'_0$ comes from Lemma \ref{lem:structure}, with input $\frac{\eps}5,2\eta$ and $\tfrac\delta2$. Set  $\rho:=\frac{\eps^2}{16M_0^2}$. 

Given $n$, $k$ and $G$, Lemma~\ref{reg:deg} yields a subgraph $G''$ of $G$ with 
 $\delta(G'')\ge (1+\tfrac{\delta}2)\frac k2$ having at least $\lceil\tfrac\delta2|G''|\rceil$ vertices of degree at least $(1+\tfrac\delta 2)k$, which has an $(\frac\eps5,2\eta)$-re\-gu\-lar partition. 
	 Apply Lemma \ref{lem:structure} to $G''$ to obtain a subgraph $G'\subseteq G''$ having an $(\varepsilon,\eta)$-regular partition with reduced graph $\mathscr R'$, which contains a cluster~$X$,  matchings $\mathscr M_V$ and $\mathscr M_W$, and a bipartite subgraph $\mathscr H=(\mathscr A,\mathscr B)$ satisfying properties $(I)-(IV)$. 
	 	
Let $T\in \mathcal T(k,\rho k)$ be given, with colour classes $A,B$. Our aim is to embed $T$ into $G'$. We may assume  $|A|\geq |B|$ and choose any $r\in B$. Apply Lemma \ref{prop:cut_linear} to $(T,r)$, with  $\beta=\frac{\eps}{|\mathscr R'|}$, to obtain a set $S\subseteq B$ with $|S|< \beta k$ and a set $\mathcal{P}$ containing all components of $T-S$. 
	
Lemma~\ref{lem:num_balance} with $\mathcal P$, $|A\cap V(P)|, |B\cap V(P)|$, $M:=k-(1-11\sqrt[4]\eps)|\bigcup V(\mathscr M_W)|$ and $\lambda:=\beta k$, yields a set $\mathcal P_1\subseteq \mathcal P$ fulfilling 
	\begin{enumerate}[(a)]
		\item\label{eq:lin:3} $\sum_{P\in \mathcal P_1}a_P\geq \dfrac{|A\setminus S|}{|B\setminus S|} \cdot \sum_{P\in \mathcal P_1}b_P> \dfrac{|B|-\beta k}{|B|} \cdot \sum_{P\in \mathcal P_1}b_P \geq \sum_{P\in \mathcal P_1}b_P - \beta k$; 
		\item\label{eq:lin:1}$ M-\beta k \le \sum_{P\in \mathcal P_1}|P|\le M$;
	\end{enumerate}
	Setting $\mathcal P_2:=\mathcal P\setminus \mathcal P_1$, from the first inequality in~\eqref{eq:lin:1} we infer that
	\begin{equation}\label{eqn:size2}
	\textstyle\sum_{P\in  \mathcal P_2}|P| 
	\le (1-10\sqrt[4]{\varepsilon})|\bigcup V(\mathscr M_W)|.
	\end{equation}
		Furthermore, by the second inequality in~\eqref{eq:lin:1} and by Lemma \ref{lem:structure}~\eqref{lem:str:VI}, 
	\begin{align}
	\textstyle \sum_{P'\in \mathcal P_1}|P'|
	 \le (1-10\sqrt[4]\eps)\textstyle\Big(\frac{|\bigcup V(\mathscr M_V)|}2+|\bigcup \mathscr A|\Big).\label{eq:J}\end{align}

We will construct an embedding  $\phi$ of $T$ into $G'$ iteratively in $|S|$ steps. In each step~$j$, we  embed some  $s_j\in S$ together with all  subtrees `below' $s_j$. We go through $S$ in an order that ensures our  embedding remains connected throughout the process, that is, we choose $s_1:=r$, and for $j\ge 2$ we choose  any yet unembedded $s_j\in S$  whose parent is already embedded. 
	Write $U_j(C)$ for the set of all unused vertices in a cluster $C$ at the beginning of step $j$.
	Four conditions will hold throughout the embedding process: 
	\begin{enumerate}[(E1)]
		\item If $j\ge 2$, the parent of $s_j$ is embedded into a vertex that is typical to~$X$.\label{p:parent}
		\item\label{E:significant} $|U_j(C)|> 5\sqrt[4]{\varepsilon}|C|$ for every cluster $C$.\label{p:good}
		\item $\bigcup  \mathcal P_2$ is  embedded into $ \bigcup V(\mathscr M_W)$,
		 \smallskip $\bigcup   \mathcal P_1\cap A$ is  embedded into $\bigcup V(\mathscr{A}\cup (\mathscr M_V \cap N(X)))$ and  
		  $\bigcup   \mathcal P_1\cap B$ is  embedded into $\bigcup V(\mathscr B\cup (V(\mathscr M_V)\setminus N(X)))$.\label{p:embJ}
		\item $\Big||U_j(C)|-|U_j( D)|\Big|\le\eps |C|$ for every edge $CD\in \mathscr M_W$.\label{p:balance}
	\end{enumerate} 	
	Now  suppose  we are at step $j\le |S|$. Choose $s_j\in S$ as detailed above.  Set 
	$$\mathscr{Y}:=\{U_j(C)\,:\,C\in  N_{\mathscr R'}(X)\cap \big(V(\mathscr M_V)\cup V(\mathscr M_W) \cup\mathscr{A}\big).$$
	Note that~(E2) ensures that every set in $\mathscr{Y}$ is significant. Since $S$ is small, we can use Fact~\ref{fact:3}~\eqref{fact:3,2} to obtain a set $X'\subset X\setminus\phi(S)$ with $|X'|\geq(1-4\sqrt{\varepsilon})|X|$  such that  
	\begin{equation}\label{eqn:deg_seed}
	\text{every $v\in X'$   is typical to at least $(1-\sqrt{\eps})|\mathscr Y|$ clusters in $\mathscr Y$.}
	\end{equation} 
	
 If $j\ge 2$, let $w$ be the image of the parent of $s_j$. By (E1), $\deg(w,X)\geq \frac{\eta }{2}|X|,$
	and hence
	$\deg(w,X')\geq\frac{\eta}{4}|X|>\beta k>|S|.$
	In particular, we can choose some vertex $v_j\in X'\cap U_j(X)$ (adjacent to $w$, if $j\ge 2$)  as $\phi(s_j)$. Now  reserve some space for  the children of $s_j$. For each  cluster $C$ such that $v_j$ is typical towards $U_j(C)$, let $C_r$ be any set of $2\varepsilon |C|+\rho k$ vertices in $N(v_j)\cap U_j(C)$. For convenience, say  
	$C\in V(\mathscr R')$ is \emph{good} if $|U_j(C)|\geq 7\sqrt[4]{\varepsilon}|C|$, and say  $CD\in E(\mathscr R')$ is \emph{good} if both $C$ and $D$ are good.

	It remains to embed  all components of $T-S$ adjacent to $s_j$ that have not been embedded yet. Let $P$ be such a component. 
 We distinguish three cases.\smallskip
	
\noindent\textbf{Case 1: } $P\in \mathcal P_1$ and there are more than $\big(1-10\sqrt[4]{\varepsilon}\big)\frac{|\bigcup V(\mathscr M_V)|}2$ unused  vertices in $\bigcup V(\mathscr M_V)$.\smallskip

In this case there are more than $\sqrt{\varepsilon}|\mathscr M_V|$ good edges in $\mathscr M_V$. Indeed, otherwise,
	\begin{align*}
	 \textstyle\big(1+10\sqrt[4]{\varepsilon}\big)\frac{|\bigcup V(\mathscr M_V)|}2&\leq \sum_{CD\in \mathscr M_V, CD\text{ good}}(|U_j(C)\cup U_j(D)|) + \sum_{CD\in \mathscr M_V, CD\text{ bad}}(|U_j(C)\cup  U_j(D)|)\\  \textstyle
	& \leq \sqrt{\varepsilon}|\mathscr M_V|\cdot  2\textstyle\frac{|G'|}{|\mathscr R'|} + |\mathscr M_V|\cdot (1+7\sqrt[4]{\varepsilon})\frac{|G'|}{|\mathscr R'|},
	\end{align*} 
	 a contradiction. So by~\eqref{eqn:deg_seed} there is a good edge $CD\in \mathscr M_V$, with $C\in N(X)$, and $v_j$  typical to $U_j(D)$. Embed the root of $P$ into $C_r$ and use Lemma~\ref{lem:T1} to embed the remaining vertices into $(U_j(C)\cup U_j(D))\setminus (C_r\cup D_r)$. In particular, all of $A\cap V(P)$ is mapped to $C$. We take care to embed  parents of vertices in $S$ into vertices that are typical to $X$. So, properties (E1)-(E4) continue to hold after this step (for (E2), recall that $CD$ is good and $|P|\le\beta k\le \sqrt[4]\eps|C|$).\smallskip
	
\noindent\textbf{Case 2: } $P\in \mathcal P_1$ and at least $\big(1-10\sqrt[4]{\varepsilon}\big)\frac{|\bigcup V(\mathscr M_V)|}2$ vertices of  $\bigcup V(\mathscr M_V)$ have been used already.

In this case, \eqref{eq:J} ensures there are at least $10\sqrt[4]{\varepsilon}|\bigcup \mathscr A|$ unused vertices in $\bigcup \mathscr A$. So, 
 there are more than $\sqrt{\varepsilon}|\mathscr A|$ good clusters in $\mathscr A$, as otherwise we reach a contradiction by calculating
$$\textstyle{10\sqrt[4]{\varepsilon}|\bigcup \mathscr{A}|}  \le  \textstyle{\sum_{C\in \mathscr{A}}|U_j(C)|}
	\leq \textstyle\sqrt{\varepsilon}|\bigcup\mathscr A|+ |\bigcup\mathscr A|7\sqrt[4]{\varepsilon}.$$                      
 By~\eqref{eqn:deg_seed},  $v_j$ is typical towards $U_j(C)$ for some  good $C\in \mathscr A$.  Moreover, there is a good cluster $D\in N_{\mathscr R'}(C)\cap \mathscr B$, as otherwise we must have already used more than 
\begin{align*}
 \textstyle(1-7\sqrt[4]{\varepsilon})\deg_{\mathscr R'}(C,\mathscr{B})\cdot \frac{|G'|}{|\mathscr R'|} & \geq \textstyle (1-7\sqrt[4]{\varepsilon})\cdot\tfrac{1}{2}\big((1+\tfrac{\delta}{4})k - |\bigcup V(\mathscr{M_W})|\big) \\
& \ge  \textstyle  \tfrac{1}{2}\big(k - (1-11\sqrt[4]{\varepsilon})|\bigcup V(\mathscr{M_W})|\big) \\
& \textstyle > \frac 12\cdot \sum_{P'\in \mathcal P_1}|P'|
\end{align*}
  vertices of $\bigcup\mathscr B$ (where the first inequality comes from Lemma~\ref{lem:structure}~\eqref{lem:str:V}, and the last one from the second inequality in~\eqref{eq:lin:1}). But this is impossible since by (E3) and by~\eqref{eq:lin:3}, we know that  $\bigcup \mathscr A$ hosts at least as many vertices from $V(\bigcup\mathcal P_1)$ as  $\bigcup \mathscr B$ does, up to an error term of~$\beta k$.
 
	Embed the root of $P$ into $C_r$ and use Lemma~\ref{lem:T1} to embed the rest of $P$  into $(U_j(C)\cup U_j(D))\setminus (C_r\cup D_r)$. Parents of vertices in $S$ are embedded into vertices that are typical to~$X$. \smallskip
	
\noindent\textbf{Case 3: } $P\in \mathcal P_2$.

	Using~\eqref{eqn:size2} and (E4) we see as above  there is a good edge $CD\in \mathscr M_W$, with~$v_j$  typical to both  $U_j(C)$ and  $U_j(D)$. Embed $P$ into $U_j(C)\cup U_j(D)$ avoiding $C_r\cup D_r$, except for the root $r_P$ of $P$. Note that we can choose into which of $C_r$ or $D_r$ we embed $r_P$, and we choose wisely so that after the embedding of $P$, (E4) still holds. As always, we embed parents of vertices in~$S$  into vertices that are typical to $X$. This finishes the embedding for Case 3, and thus the proof of Lemma~\ref{prop:connected_linear}.
\end{proof}

\bibliographystyle{acm}
\bibliography{trees}
\appendix{
\section{Concentration lemmas}\label{sec:concentration}

In this appendix we prove two results on the concentration of a given function around its mean value. Given $N\in\mathbb{N}$ and a function $f:[N]\to\mathbb{R}$, we write 
$$\|f\|_\infty=\max_{n\in[N]}|f(n)|$$ for the infinity norm of $f$. If $\mu$ is a probability measure on $[N]$ then, as usual, $$\EE_\mu(f)=\sum_{n\in [N]}f(n)\mu(n)$$
denotes the expectation of $f$ under $\mu$, and if $\mu$ is the uniform probability we write 
$$\mathbb{E}_{n\in[N]}f(n)=\frac{1}{N}\sum_{n\in [N]}f(n).$$

\begin{lemma}\label{lem:con}
	Let $N\in\mathbb N$, $t\in\mathbb R$ and  $\eps\in(0,1)$. Let $\mu$ be a probability measure on $[N]$ and let $f:[N]\to \mathbb{R}_+$  satisfying $\sqrt\eps \|f\|_\infty<t\le \EE_\mu (f).$ Then at least one of the following holds 
	\begin{enumerate}[(i)]
		\item $\mu(\{n:f(n)>(1+\sqrt{\varepsilon})t\})\ge \varepsilon$, or 
		\item $\mu(\{n:f(n)>(1-\sqrt[4]{\varepsilon}) t\})\ge1-\sqrt[4]{\varepsilon}$.\end{enumerate}
\end{lemma}

	\begin{proof}
		Let $A$ be the set of all $n\in[N]$ with $f(n)>(1+\sqrt{\varepsilon})t$ and set $B:=[N]\setminus A$. Suppose that (i) does not hold. Then $\mu(A)\le \varepsilon$, and therefore,
		\begin{align}
		\sum_{n\in B}\mu(n)f(n) =\EE_\mu (f) -\sum_{n\in A}\mu(n)f(n) 
	 \ge t-\mu(A)\|f\|_\infty 
		 \ge (1-\sqrt{\varepsilon})t.\label{eq:02}
		\end{align} 
		Let $B_1$ be the set of all $n\in B$  such that $f(n)<(1+\sqrt{\varepsilon}-2\sqrt[4]{\varepsilon}) t$, and set $B_2:=B\setminus B_1$. From \eqref{eq:02} and the definition of $B$ we deduce that 
		\begin{align*}
			(1-\sqrt{\varepsilon})t 
			\le(1+\sqrt{\varepsilon})t\cdot \mu(B)-2\sqrt[4]{\varepsilon}t\cdot \mu(B_1) 
			\le(1+\sqrt{\varepsilon})t-2\sqrt[4]{\varepsilon}t\cdot \mu(B_1),
		\end{align*}
		and hence, $\mu(B_1)\le  \sqrt[4]{\varepsilon}$. Therefore,   $\mu(A\cup B_2)\ge 1-\mu(B_1)\ge 1-\sqrt[4]{\varepsilon},$ which implies (ii).\end{proof}	}
	As a corollary of Lemma~\ref{lem:con} we get the following useful result. 
	
\begin{lemma}\label{lem:con2} 
	Let $N\in\mathbb N$, and let $\eps\in(0,\frac 12)$. Let $f:[N]\to \mathbb{R}_+$ be a function and let $t>0$ such that $t\le \EE_{n\in [N]} f(n)$ and $\|f\|_{\infty}\le(1+\eps)t$. Then $f(n)\ge (1-\sqrt{\eps})t$ for every $n$ in a set of size at least $(1-\sqrt{\eps})N$.
\end{lemma}
%
	\end{document}